\documentclass[12pt,reqno]{amsart}

\usepackage{amssymb}

\usepackage{mathrsfs}

\DeclareMathAlphabet{\mathpzc}{OT1}{pzc}{m}{it}

\usepackage[all]{xy}

\entrymodifiers={+!!<0pt,\fontdimen22\textfont2>}

\def\sA{\mathsf{A}}
\def\sB{\mathsf{B}}

\def\sD{\mathsf{D}}

\def\sK{\mathsf{K}}

\def\sT{\mathsf{T}}
\def\sU{\mathsf{U}}
\def\sV{\mathsf{V}}
\def\sW{\mathsf{W}}
\def\sX{\mathsf{X}}
\def\sY{\mathsf{Y}}

\def\adots{\mathinner{\mkern1mu\raise1.0pt\vbox{\kern7.0pt\hbox{.}}\mkern2mu\raise4.0pt\hbox{.}\mkern2mu\raise7.0pt\hbox{.}\mkern1mu}}

\def\b{\operatorname{b}}

\def\D{\sD}

\def\Df{\D^{\operatorname{f}}}

\def\Flat{\operatorname{Flat}}

\def\GMor{\mathsf{GMor}}
\def\GProj{\mathsf{GProj}}

\def\Hom{\operatorname{Hom}}

\def\inc{\operatorname{inc}}

\def\inj{\operatorname{inj}}
\def\Inj{\operatorname{Inj}}

\def\LTensor{\stackrel{\operatorname{L}}{\otimes}}

\def\Mod{\mathsf{Mod}}

\def\Morph{\mathsf{Mor}}
\def\opp{\operatorname{op}}

\def\prj{\operatorname{prj}}
\def\Prj{\operatorname{Prj}}
\def\prod{\operatorname{prod}}

\def\RHom{\operatorname{RHom}}

\def\s{\operatorname{s}}

\def\tac{\operatorname{tac}}

\numberwithin{equation}{section}

\newtheorem{Lemma}{Lemma}[section]
\newtheorem{Theorem}[Lemma]{Theorem}
\newtheorem{Proposition}[Lemma]{Proposition}
\newtheorem{Corollary}[Lemma]{Corollary}

\theoremstyle{definition}
\newtheorem{Definition}[Lemma]{Definition}
\newtheorem{Setup}[Lemma]{Setup}

\newtheorem{Remark}[Lemma]{Remark}

\begin{document}

\setlength{\parindent}{0pt}
\setlength{\parskip}{7pt}

\title[Symmetric Auslander and Bass categories]{Symmetric Auslander
and Bass categories} 

\author{Peter J\o rgensen}
\address{School of Mathematics and Statistics,
Newcastle University, Newcastle upon Tyne NE1 7RU, United Kingdom}
\email{peter.jorgensen@ncl.ac.uk}
\urladdr{http://www.staff.ncl.ac.uk/peter.jorgensen}

\author{Kiriko Kato}
\address{Department of Mathematics and Information Sciences, Osaka
Prefecture University, Japan}
\email{kiriko@mi.s.osakafu-u.ac.jp}





\begin{abstract} 

We define the symmetric Auslander category $\sA^{\s}(R)$ to consist
of complexes of projective modules whose left- and right-tails are
equal to the left- and right-tails of totally acyclic complexes of
projective modules.

The symmetric Auslander category contains $\sA(R)$, the ordinary
Auslander category.  It is well known that $\sA(R)$ is intimately
related to Gorenstein projective modules, and our main result is that
$\sA^{\s}(R)$ is similarly related to what can reasonably be called
Gorenstein projective homomorphisms.  Namely, there is an equivalence
of triangulated categories
\[
  \underline{\GMor}(R)
  \stackrel{\simeq}{\rightarrow} \sA^{\s}(R) / \sK^{\b}(\Prj\,R)  
\]
where $\underline{\GMor}(R)$ is the stable category of Gorenstein
projective objects in the abelian category $\Morph(R)$ of
homomorphisms of $R$-modules.

This result is set in the wider context of a theory for $\sA^{\s}(R)$
and $\sB^{\s}(R)$, the symmetric Bass category which is defined
dually.

\end{abstract}

\maketitle

\setcounter{section}{-1}
\section{Introduction}
\label{sec:introduction}

Let $R$ be a commutative noetherian ring with a dualizing complex $D$.
Such complexes were introduced in \cite[chp.\ V]{H} where it was also
shown that the functor $\RHom_R(-,D)$ is a contravariant
autoequivalence of $\Df(R)$, the finite derived category of $R$.

Some time later, it was shown in \cite[sec.\ 3]{AF} that by
restricting to certain subcategories $\sA(R)$ and $\sB(R)$ of the
derived category $\sD(R)$, the functors $D \LTensor_R -$ and
$\RHom_R(D,-)$ become quasi-inverse covariant equivalences
\[
  \xymatrix{
    \sA(R) \ar@<1ex>[rrr]^-{D \LTensor_R -}
    & & & \sB(R). \ar@<1ex>[lll]^-{\RHom_R(D,-)}
           }
\]
The categories $\sA(R)$ and $\sB(R)$ are known as the Auslander and
Bass categories of $R$.  The precise definition is given in Remark
\ref{rmk:AB} below, but note that $\sA(R)$ and $\sB(R)$ contain the
bounded complexes of projective, respectively injective, modules.

This paper introduces the symmetric Auslander category $\sA^{\s}(R)$
and the symmetric Bass category $\sB^{\s}(R)$ which contain $\sA(R)$,
respectively $\sB(R)$, as full subcategories.  While $\sA(R)$ enjoys a
strong relation to Gorenstein projective modules, our main result is
that $\sA^{\s}(R)$ has a similarly close relation to {\em
homomorphisms} of Gorenstein projective modules.

This result is set in the wider context of a theory which shows that
the two new categories inhabit a universe with strong symmetry
properties.

\medskip
\noindent
{\em Background on Auslander and Bass categories. }
Recall that the Auslander category $\sA(R)$ can be characterized in
terms of totally acyclic complexes of projective modules.  Such a
complex $P$ consists of projective modules, is exact, and has the
property that $\Hom_R(P,Q)$ is exact for each projective module $Q$.
It was proved in \cite[sec.\ 4]{CFH} that a complex is in $\sA(R)$ if
and only if its homology is bounded and the left-tail of its
projective resolution is equal to the left-tail of a totally
acyclic complex of projective modules (all differentials point to the
right).

The left-tails of totally acyclic complexes of projective modules are
precisely the projective resolutions of so-called Gorenstein
projective mo\-du\-les; this is immediate from the definition of a
Gorenstein projective module as a cycle module of a totally acyclic
complex of projectives, see \cite{EJ}.  This leads to the expectation
that if we remove from $\sA(R)$ a suitable ``finite'' part, leaving
only the tails of projective resolutions, then we should get a
category of Gorenstein projective modules.

Indeed, the homotopy category $\sK^{\b}(\Prj\,R)$ of bounded complexes
of projective modules can be viewed as a subcategory of $\sA(R)$, and
we can remove it by forming the Verdier quotient $\sA(R) /
\sK^{\b}(\Prj\,R)$.  On the other hand, the Gorenstein projective
modules form a Frobenius category $\GProj(R)$, and there is a stable
category $\underline{\GProj}(R)$ obtained by dividing out
homomorphisms which factor through projective modules.  It is not
hard to show that there is an equivalence of triangulated categories
\begin{equation}
\label{equ:A}
  \underline{\GProj}(R)
  \stackrel{\simeq}{\rightarrow} \sA(R) / \sK^{\b}(\Prj\,R).
\end{equation}

\medskip
\noindent
{\em Symmetric Auslander and Bass categories. }
The main result of this paper is a higher analogue of the above
phenomenon.  Let $\sK(\Prj\,R)$ be the homotopy category of complexes
of projective modules.  We define the symmetric Auslander category
$\sA^{\s}(R)$ to be the full subcategory of $\sK(\Prj\,R)$ consisting
of complexes whose left- and right-tails are equal to the left-
and right-tails of totally acyclic complexes of projective modules.

Our main result is the following.

\smallskip
\noindent
{\bf Theorem A. }
{\em
There is an equivalence of triangulated categories
\[
  \underline{\GMor}(R)
  \stackrel{\simeq}{\rightarrow} \sA^{\s}(R) / \sK^{\b}(\Prj\,R).
\]
}

Here $\underline{\GMor}(R)$ is the stable category of Gorenstein
projective objects in $\Morph(R)$, the abelian category of
homomorphisms of $R$-modules.  Note that there is an equivalence of
categories between $\Morph(R)$ and $\Mod\, T_2(R)^{\opp}$, the category
of right-modules over the upper triangular matrix ring $T_2(R)$; cf.\
\cite{Aus}.  This implies that $\underline{\GMor}(R)$ is equivalent to
the stable category of Gorenstein projective right-modules over
$T_2(R)$.

On the other hand, we will show that the objects in
$\underline{\GMor}(R)$ are precisely the injective homomorphisms
between Gorenstein projective $R$-modules which have Gorenstein
projective cokernels.  Hence, whereas the Auslander category $\sA(R)$
is related to Gorenstein projective mo\-du\-les via equation
\eqref{equ:A}, the symmetric Auslander category $\sA^{\s}(R)$ is
si\-mi\-lar\-ly related to {\em homomorphisms} of Gorenstein
projective modules via Theorem A.

To prove the theorem, we develop a theory for the symmetric Auslander
and Bass categories.  One of the highlights is that $\sA^{\s}(R)$ is,
indeed, a highly sym\-me\-tric object.  Namely, the quotient
$\sA^{\s}(R) / \sK^{\b}(\Prj\,R)$ permits a so-called triangle of
recollements $(\sU,\sV,\sW)$ as introduced in \cite{IKM}.  This means
that $\sU$, $\sV$, $\sW$ are full subcategories of $\sA^{\s}(R) /
\sK^{\b}(\Prj\,R)$, and that each of
\[
  (\sU \, , \, \sV) \; , \; (\sV \, , \, \sW) \; , \; (\sW \, , \, \sU)
\]
is a stable t-structure.  It is not obvious, even in principle, that
such a configuration is possible, but we show that
\begin{align}
\label{equ:UVW}
\nonumber
  \sU & = \sA(R) / \sK^{\b}(\Prj\,R), \\
  \sV & = \sK_{\tac}(\Prj\,R), \\
\nonumber
  \sW & = S(\sB(R)) / \sK^{\b}(\Prj\,R)
\end{align}
work, where $\sK_{\tac}(\Prj\,R)$ is the full subcategory of
$\sK(\Prj\,R)$ consisting of totally acyclic complexes and $S$ is a
certain functor introduced in \cite[sec.\ 4]{IK}.

There are also several other results, among them the following.

\smallskip
\noindent
{\bf Theorem B. }
{\em
There are quasi-inverse equivalences of triangulated
ca\-te\-go\-ri\-es 
\[
  \xymatrix{
    \sA^{\s}(R) \ar@<1ex>[r]
    & \sB^{\s}(R). \ar@<1ex>[l]
           }
\]
}

Let $\sK^{(\b)}(\Prj\,R)$ denote the full subcateogry of
$\sK(\Prj\,R)$ consisting of complexes with bounded homology.

\smallskip
\noindent
{\bf Theorem C. }
{\em
There are inclusions
\[
  \sA(R) \subseteq \sA^{\s}(R) \subseteq \sK^{(\b)}(\Prj\,R).
\]
The first inclusion is an equality if and only if each Gorenstein
projective $R$-module is projective.

The second inclusion is an equality if and only if $R$ is a Gorenstein
ring.
}

Thus, the property that $\sA^{\s}(R)$ is minimal, respectively
maximal, cha\-rac\-te\-ri\-ses two interesting classes of rings.

Let us remark on two important sources of ideas for this paper.
First, \cite{IKM} originated the notion of a triangle of recollements
and used it to get a version of Theorem A for finitely generated
modules when $R$ is a Gorenstein ring.  The present paper can be
viewed as extending these ideas.  Secondly, while it is not obvious
from the description above, we make extensive use of the machinery
developed in \cite{IK} for homotopy categories of complexes of
projective, respectively, injective modules and their relation to
Auslander and Bass categories.

The paper is organised as follows: Section \ref{sec:background}
briefly sketches the definitions and results we will use; most of them
come from \cite{IK}.  Section \ref{sec:AsBs} proves Theorems B and C
above (Theorems \ref{thm:I} and \ref{thm:II}) and establishes the
existence of the triangle of recollements described by equation
\eqref{equ:UVW} (Theorem \ref{thm:A_Ktac_SB2}).  Section
\ref{sec:Morph} studies the category of homomorphisms $\Morph(R)$ and
its Gorenstein projective objects, and culminates in the proof of
Theorem A (Theorem \ref{thm:main}).

\bigskip
\noindent
{\em Acknowledgement. }
We thank Srikanth Iyengar for comments on a preliminary version which
led to Remark \ref{rmk:Sri}.

The first author wishes to express his sincere gratitude for the
hospitality of the second author and the Visitor Programme operated by
the Graduate School of Science at Osaka Prefecture University.  The
second author was partially supported by JSPS grant 18540044.

\section{Background}
\label{sec:background}

This section recalls the tools we will use; most of them come from
\cite{IK}. 

\begin{Setup}
\label{set:blanket}
Throughout, $R$ is a commutative noetherian ring with a dualizing
complex $D$ which is assumed to be a bounded complex of injective
modules.
\end{Setup}

Dualizing complexes were introduced in \cite{H}, but see e.g.\
\cite[sec.\ 1]{CFH} for a contemporary introduction.

\begin{Remark}
There are homotopy categories $\sK(\Prj\,R)$ and $\sK(\Inj\,R)$ of
complexes of projective, respectively, injective modules.  They have
several important triangulated subcategories:

The subcategories of bounded complexes are denoted by
$\sK^{\b}(\Prj\,R)$ and $\sK^{\b}(\Inj\,R)$.  The subcategories of
complexes with bounded homology are denoted by $\sK^{(\b)}(\Prj\,R)$
and $\sK^{(\b)}(\Inj\,R)$.

The subcategories of K-projective, respectively, K-injective complexes
are denoted by $\sK_{\prj}(R)$ and $\sK_{\inj}(R)$; see \cite{S}.

The subcategories of totally acyclic complexes are denoted
$\sK_{\tac}(\Prj\,R)$ and $\sK_{\tac}(\Inj\,R)$.  Complexes $X$ in
$\sK(\Prj\,R)$ and $Y$ in $\sK(\Inj\,R)$ are called totally acyclic if
they are exact and $\Hom_R(X,P)$ and $\Hom_R(I,Y)$ are exact for each
projective module $P$ and each injective module $I$.
\end{Remark}

\begin{Remark}
\label{rmk:inc}
Consider the subcategories $\sK_{\prj}(R) \subseteq \sK(\Prj\,R)$ and
$\sK_{\inj}(R) \subseteq \sK(\Inj\,R)$.  By \cite[sec.\ 7]{IK}, the
inclusion functors, which we will denote by $\inc$, are parts of
adjoint pairs of functors,
\[
  \xymatrix{
    \sK_{\prj}(R) \ar@<1ex>[r]^{\inc}
    & \sK(\Prj\,R) \ar@<1ex>[l]^{p}
           }
\;\;\mbox{and}\;\;
  \xymatrix{
    \sK_{\inj}(R) \ar@<-1ex>[r]_{\inc}
    & \sK(\Inj\,R). \ar@<-1ex>[l]_{i}
           }
\]
In the terminology of \cite[chp.\ 9]{Neemanbook}, the existence of the
right adjoint $p$ places us in a situation of Bousfield localization,
and accordingly, the counit morphism of the adjoint pair $(\inc,p)$
can be completed to a distinguished triangle
\[
  pX
  \stackrel{\epsilon_P}{\longrightarrow} X
  \longrightarrow aX
  \longrightarrow
\]
which depends functorially on $X$.  Both $p$ and $a$ are triangulated
functors.  Dually, the unit morphism of the adjoint pair $(i,\inc)$
can be completed to a distinguished triangle
\[
  bY
  \longrightarrow Y
  \stackrel{\eta_Y}{\longrightarrow} iY
  \longrightarrow
\]
which depends functorially on $Y$.
\end{Remark}

\begin{Remark}
\label{rmk:ST}
By \cite[thm.\ 4.2]{IK} there are quasi-inverse equivalences of
categories
\[
  \xymatrix{
    \sK(\Prj\,R) \ar@<1ex>[r]^-{T}
    & \sK(\Inj\,R) \ar@<1ex>[l]^-{S}
           }
\]
where $T(-) = D \otimes_R -$ and $S = q \circ \Hom_R(D,-)$.  The
functor $q$ is right-adjoint to the inclusion $\sK(\Prj\,R)
\rightarrow \sK(\Flat\,R)$ where $\sK(\Flat\,R)$ is the homotopy
category of complexes of flat modules.
\end{Remark}

\begin{Remark}
\label{rmk:AB}
Let us recall the following from \cite{AF}.  The derived category
$\sD(R)$ supports an adjoint pair of functors
\[
  \xymatrix{
    \sD(R) \ar@<1ex>[rrr]^-{D \LTensor_R -}
    & & & \sD(R). \ar@<1ex>[lll]^-{\RHom_R(D,-)}
           }
\]
The Auslander category of $R$ is the triangulated subcategory defined
in terms of the unit $\eta$ by
\[
  \sA(R)
  = \Biggl\{ X \in \sD(R)
    \Bigg|
      \begin{array}{l}
        \mbox{\small{$X$ and $D \LTensor_R X$ have bounded homology;}} \\
        \mbox{\small{$X \stackrel{\eta_X}{\longrightarrow} \RHom_R(D,D \LTensor_R X)$ is an isomorphism}} \\
      \end{array}
    \Biggr\}
\]
and the Bass category of $R$ is the triangulated subcategory defined
in terms of the counit $\epsilon$ by
\[
  \sB(R)
  = \Biggl\{ Y \in \sD(R)
    \Bigg|
      \begin{array}{l}
        \mbox{\small{$Y$ and $\RHom_R(D,Y)$ have bounded homology;}} \\
        \mbox{\small{$D \LTensor_R \RHom_R(D,Y) \stackrel{\epsilon_Y}{\longrightarrow} Y$ is an isomorphism}} \\
      \end{array}
    \Biggr\}.
\]
The functors $D \LTensor_R -$ and $\RHom_R(D,-)$ restrict to
quasi-inverse equivalences between $\sA(R)$ and $\sB(R)$.

The canonical functors $\sK_{\prj}(R) \rightarrow \D(R)$ and
$\sK_{\inj}(R) \rightarrow \D(R)$ are e\-qui\-va\-len\-ces, and this
permits us to view $\sA(R)$ as a full subcategory of $\sK_{\prj}(R)$
and hence of $\sK(\Prj\,R)$, and $\sB(R)$ as a full subcategory of
$\sK_{\inj}(R)$ and hence of $\sK(\Inj\,R)$.  As such, the adjoint
functors
\[
  \xymatrix{
    \sK_{\prj}(R) \ar@<1ex>[r]^-{iT}
    & \sK_{\inj}(R) \ar@<1ex>[l]^-{pS}
           }
\]
restrict to a pair of quasi-inverse equivalences between $\sA(R)$ and
$\sB(R)$ by \cite[prop.\ 7.2]{IK}.
\end{Remark}

See \cite[sec.\ 1]{CFH} for an alternative review of Auslander and
Bass categories.

\begin{Definition}
Let $\sT$ be a triangulated category.  A stable t-structure on $\sT$
is a pair of full subcategories $(\sU,\sV)$ such that
\begin{enumerate}

  \item  $\Sigma\sU = \sU$, $\Sigma\sV = \sV$.

\smallskip

  \item  $\Hom_{\sT}(\sU,\sV) = 0$.

\smallskip

  \item  For each $T$ in $\sT$ there exist $U$ in $\sU$ and $V$ in
$\sV$ and a distinguished triangle $U \rightarrow T \rightarrow V
\rightarrow$. 

\end{enumerate}
A triangle of recollements in $\sT$ is a triple $(\sU,\sV,\sW)$ such
that each of $(\sU,\sV)$, $(\sV,\sW)$, $(\sW,\sU)$ is a stable
t-structure.

Let $\sT'$ be another triangulated category with a triangle of
recollements $(\sU',\sV',\sW')$ and let $F : \sT \rightarrow \sT'$ be
a triangulated functor.  We say that $F$ sends $(\sU,\sV,\sW)$ to
$(\sU',\sV',\sW')$ if $F(\sU) \subseteq \sU'$, $F(\sV) \subseteq
\sV'$, $F(\sW) \subseteq \sW'$.
\end{Definition}

\section{Symmetric Auslander and Bass categories}
\label{sec:AsBs}

This section develops a theory of symmetric Auslander and Bass
ca\-te\-go\-ri\-es.  It proves Theorems B and C from the Introduction,
and establishes the existence of the triangle of recollements
described by equation \eqref{equ:UVW} (Theorems \ref{thm:I},
\ref{thm:II}, and \ref{thm:A_Ktac_SB2}).

For the rest of the paper, an unadorned $\sK$ stands for
$\sK(\Prj\,R)$.  We combine this in an obvious way with various
embellishments to form $\sK^{\b}$, $\sK^{(\b)}$, $\sK_{\prj}$, and
$\sK_{\tac}$.  Likewise, unadorned categories such as $\sA$, $\sB$,
and $\sD$ stand for $\sA(R)$, $\sB(R)$, and $\sD(R)$.

In the following definition, $\sX * \sY$ denotes the full subcategory
of objects $C$ which sit in distinguished triangles $X \rightarrow C
\rightarrow Y \rightarrow$ with $X$ in $\sX$ and $Y$ in $\sY$.

\begin{Definition}
The {\em symmetric Auslander category} $\sA^{\s}$ and the {\em
symmetric Bass category} $\sB^{\s}$ of $R$ are the full subcategories
of $\sK(\Prj\,R)$ and $\sK(\Inj\,R)$ defined by
\[
  \sA^{\s} = S(\sB) * \sA
  \;\; \mbox{and} \;\;
  \sB^{\s} = \sB * T(\sA)
\]
where $S$ and $T$ are the functors from \cite{IK} described in Remark
\ref{rmk:ST}.
\end{Definition}

\begin{Remark}
By \cite[thm.\ 4.1]{CFH}, the subcategory $\sA$ of $\sK$
consists of complexes isomorphic to right-bounded complexes of
projective mo\-du\-les whose left-tail is equal to the left-tail of a
complete projective resolution. 

Using the theory of \cite{IK}, one can show that similarly, $S(\sB)$
consists of complexes isomorphic to left-bounded complexes of
projective mo\-du\-les whose right-tail is equal to the right-tail of
a complete projective resolution.

From this it follows that $\sA^{\s}$ consists of complexes isomorphic
to complexes of projective modules both of whose tails are equal to
the tails of complete projective resolutions.

Similar remarks apply to $\sB^{\s}$, and this is one of the reasons
for the terminology ``symmetric Auslander and Bass categories''.
\end{Remark}

\begin{Remark}
The following lemma and most of the other results in this section will
only be given for $\sA^{\s}$, but there are dual versions for
$\sB^{\s}$ with similar proofs.
\end{Remark}

\begin{Lemma}
\label{lem:Kiriko}
Let $C$ be in $\sK$.  Then $C$ is in $\sA^{\s}$ if and only
if the following conditions are satisfied.
\begin{enumerate}

  \item  $C$ and $TC$ have bounded homology.

\smallskip

  \item  The mapping cone of $pC
\stackrel{\epsilon_C}{\longrightarrow} C$ is totally acyclic.

\smallskip

  \item  The mapping cone of $TC \stackrel{\eta_{TC}}{\longrightarrow}
iTC$ is totally acyclic.

\end{enumerate}
\end{Lemma}

\begin{proof}
``Only if'': Suppose that $C$ is in $\sA^{\s}$.  By definition, there
is a distinguished triangle
\[
  SB \rightarrow C \rightarrow A \rightarrow
\]
in $\sK$ with $B$ in $\sB$ and $A$ in $\sA$.  All of $SB$, $A$, $TSB
\cong B$, and $TA$ have bounded homology, so the same is true for $C$
and $TC$, proving condition (i).

By Remark \ref{rmk:inc}, the distinguished triangle induces the
following commutative diagram where each row and each column is a
distinguished triangle.
\[
  \xymatrix{
    pSB \ar[r] \ar[d]_{\epsilon_{SB}} & pC \ar[r] \ar[d]_{\epsilon_C} & pA \ar[r] \ar[d]_{\epsilon_A} & {} \\
    SB \ar[r] \ar[d] & C \ar[r] \ar[d] & A \ar[r] \ar[d] & {} \\
    aSB \ar[r]_{\alpha} \ar[d] & aC \ar[r] \ar[d] & aA \ar[r] \ar[d] & {} \\
    {} & {} & {} &
           }
\]
Since $A$ is K-projective, $\epsilon_A$ is an isomorphism.  Hence $aA$
is zero so $\alpha$ is an isomorphism.  But $B$ is in $\sB$ so $aSB$
is totally acyclic by \cite[prop.\ 7.4]{IK}, and so $aC$ is totally
acyclic, proving condition (ii).  A similar argument proves condition
(iii).

``If'': Suppose that conditions (i) through (iii) hold.  Hard
truncation gives a distinguished triangle
\[
  C^{\geq 0} \rightarrow C \rightarrow C^{<0} \rightarrow
\]
in $\sK$.  We aim to show that $C^{\geq 0}$ is in $S(\sB)$ and that
$C^{<0}$ is in $\sA$ whence $C$ is in $\sA^{\s}$.

Set
\[
  B = T(C^{\geq 0}) = D \otimes_R C^{\geq 0}
\]
so $SB = ST(C^{\geq 0}) \cong C^{\geq 0}$.  Since $C^{\geq 0}$ is a
left-bounded complex of projective modules and $D$ is a bounded
complex of injective modules, $B$ is a left-bounded complex of
injective modules.  In particular, it is K-injective.

Since $D$ is bounded, the complexes $B$ and $TC = D \otimes_R C$ agree
in high cohomological degrees.  But $B$ is left-bounded and $TC$ has
bounded homology by condition (i), so it follows that $B$ has bounded
ho\-mo\-lo\-gy.  Also, $B$ is K-injective so $\RHom_R(D,B)$ can be
computed as $\Hom_R(D,B)$, but
\[
  \Hom_R(D,B)
  \stackrel{\rm (a)}{\simeq} q \circ \Hom_R(D,B)
  = SB
  \cong C^{\geq 0}
\]
where the quasi-isomorphism (a) is by \cite[thm.\ 2.7]{IK}.  Since the
homology of $C^{\geq 0}$ is bounded, so is the homology of
$\RHom_R(D,B)$. 

As above, the distinguished triangle induces the following commutative
diagram where each row and each column is a distinguished triangle.
\[
  \xymatrix{
    pC^{\geq 0} \ar[r] \ar[d]_{\epsilon_{C^{\geq 0}}} & pC \ar[r] \ar[d]_{\epsilon_C} & pC^{<0} \ar[r] \ar[d]_{\epsilon_{C^{<0}}} & {} \\
    C^{\geq 0} \ar[r] \ar[d] & C \ar[r] \ar[d] & C^{<0} \ar[r] \ar[d] & {} \\
    aC^{\geq 0} \ar[r]_{\beta} \ar[d] & aC \ar[r] \ar[d] & aC^{<0} \ar[r] \ar[d] & {} \\
    {} & {} & {} &
           }
\]
Since $C^{<0}$ is a right-bounded complex of projective modules it is
K-projective and so $\epsilon_{C^{<0}}$ is an isomorphism.  Hence
$aC^{<0}$ is zero so $\beta$ is an isomorphism.  But $aC$ is totally
acyclic by condition (ii), and so $aC^{\geq 0}$ is totally acyclic.
Since $SB \cong C^{\geq 0}$, it follows from \cite[prop.\ 7.4]{IK} that
$B$ is in $\sB$ and so $C^{\geq 0}$ is in $S(\sB)$.

A similar argument proves that $C^{<0}$ is in $\sA$.
\end{proof}

\begin{Proposition}
\label{pro:subcats}
The category $\sA^{\s}$ is a triangulated subcategory of $\sK$,
and there are inclusions of triangulated subcategories
\[
  \sK_{\tac} \subseteq \sA^{\s} \subseteq \sK^{(\b)}.
\]
\end{Proposition}

\begin{proof}
It is well known that $\sK_{\tac}$ and $\sK^{(\b)}$ are triangulated
subcategories of $\sK$.

Conditions (i) through (iii) of Lemma \ref{lem:Kiriko} respect mapping
cones, so $\sA^{\s}$ is a triangulated subcategory of $\sK$.

The second inclusion of the proposition is immediate from Lemma
\ref{lem:Kiriko}(i), and the first one follows from Lemma
\ref{lem:Kiriko}(i)--(iii) combined with the fact that $T$ sends
totally acyclic complexes to totally acyclic complexes by \cite[prop.\
5.9(1)]{IK}.
\end{proof}

\begin{Remark}
\label{rmk:Sri}
We owe the following observations based on Lemma \ref{lem:Kiriko} to
Srikanth Iyengar.

The Auslander and Bass categories $\sA$ and $\sB$ also exist in
versions $\widehat{\sA}$ and $\widehat{\sB}$ without boundedness
conditions \cite[7.1]{IK}.  With small modifications, the proof of
Lemma \ref{lem:Kiriko} shows that membership of
$S(\widehat{\sB})*\widehat{\sA}$ is characterised by conditions (ii)
and (iii) of the Lemma.

It is immediate from Lemma \ref{lem:Kiriko} that $\sA * S(\sB)$ is
contained in $\sA^{\s} = S(\sB) * \sA$.  This is a bit surprising
since one would not normally expect any inclusion between categories
of the form $\sX * \sY$ and $\sY * \sX$.

We do not know if $\sA * S(\sB)$ is triangulated, but it will often be
considerably smaller than $S(\sB) * \sA$ since $\sK_{\tac}$ is
contained in $S(\sB) * \sA$ by Proposition \ref{pro:subcats} while it
is easy to show that the intersection of $\sA * S(\sB)$ with
$\sK_{\tac}$ is zero.
\end{Remark}

\begin{Theorem}
\label{thm:I}
The functors $T$ and $S$ restrict to quasi-inverse
e\-qui\-va\-len\-ces of triangulated categories
\[
  \xymatrix{
    \sA^{\s} \ar@<1ex>[r]^-{T}
    & \sB^{\s}. \ar@<1ex>[l]^-{S}
           }
\]
\end{Theorem}

\begin{proof}
This is immediate from the definition of $\sA^{\s}$ and $\sB^{\s}$
because $T$ and $S$ are quasi-inverse equivalences of triangulated
categories.
\end{proof}

\begin{Theorem}
\label{thm:A_Ktac_SB}
\begin{enumerate}

  \item  The category $\sA^{\s}$ has stable t-structures
\[
  (\sA,\sK_{\tac}(\Prj\,R))
  \;\; \mbox{and} \;\;
  (\sK_{\tac}(\Prj\,R),S(\sB)).
\]

\smallskip

  \item  The category $\sB^{\s}$ has stable t-structures
\[
  (\sK_{\tac}(\Inj\,R),\sB)
  \;\; \mbox{and} \;\;
  (T(\sA),\sK_{\tac}(\Inj\,R)).
\]

\end{enumerate}
\end{Theorem}

\begin{proof}
The first of the stable t-structures in part (i) can be established as
follows.

The category $\sA^{\s}$ contains $\sA$ by definition and $\sK_{\tac}$
by Proposition \ref{pro:subcats}.  Each $A$ in $\sA$ is K-projective,
so a morphism $A \rightarrow U$ with $U$ in $\sK_{\tac}$ is zero.

Existence of the first stable t-structure will thus follow if we can
prove $\sA^{\s} = \sA * \sK_{\tac}$.

For $C$ in $\sA^{\s}$, there is a distinguished triangle $SB
\longrightarrow C \longrightarrow A \longrightarrow$ with $B$ in $\sB$
and $A$ in $\sA$.  Turning the triangle gives a distinguished triangle
$\Sigma^{-1}A \stackrel{\alpha}{\longrightarrow} SB \longrightarrow C
\longrightarrow A$.

There is also a distinguished triangle $pSB
\stackrel{\epsilon_{SB}}{\longrightarrow} SB \longrightarrow U
\longrightarrow$ and $U$ is totally acyclic by \cite[prop.\ 7.4]{IK}.
Since $\Sigma^{-1}A$ is in $\sA$, each morphism $\Sigma^{-1}A
\rightarrow U$ is zero, and hence $\alpha$ lifts through
$\epsilon_{SB}$. 

By the octahedral axiom, there is hence a commutative diagram in which
each row and each column is a distinguished triangle,
\[
  \xymatrix{
    \Sigma^{-1}pSB \ar[r] \ar[d] & \Sigma^{-1}SB \ar[r] \ar[d] & \Sigma^{-1}U \ar[r] \ar[d] & pSB \ar[d] \\
    \Sigma^{-1}A^{\prime} \ar[r] \ar[d] & 0 \ar[r] \ar[d] & A^{\prime} \ar@{=}[r] \ar[d] & A^{\prime} \ar[d] \\
    \Sigma^{-1}A \ar[r]^{\alpha} \ar[d] & SB \ar[r] \ar@{=}[d] & C \ar[r] \ar[d] & A \ar[d] \\
    pSB \ar[r]_{\epsilon_{SB}} & SB \ar[r] & U \ar[r] & \Sigma pSB.
           }
\]
Since $B$ is in $\sB$, the object $pSB$ is in $\sA$ by \cite[prop.\
7.2]{IK}; see Remark \ref{rmk:AB}.  Since $A$ is also in $\sA$, it
follows that $A^{\prime}$ is in $\sA$.  So the third column of the
above diagram shows $\sA^{\s} = \sA * \sK_{\tac}$, proving existence
of the first stable t-structure in the theorem.

The first of the stable t-structures in part (ii) follows by an
analogous argument using \cite[prop.\ 7.3]{IK} instead of \cite[prop.\
7.2]{IK}.

The second stable t-structure in part (i) is obtained by applying $S$
to the first stable t-structure in part (ii).  The second stable
t-structure in part (ii) is obtained by applying $T$ to the first
stable t-structure in part (i).
\end{proof}

\begin{Theorem}
\label{thm:II}
There are inclusions
\[
  \sA \subseteq \sA^{\s} \subseteq \sK^{(\b)}.
\]
The first inclusion is an equality if and only if each Gorenstein
projective $R$-module is projective.

The second inclusion is an equality if and only if $R$ is a Gorenstein
ring.
\end{Theorem}

\begin{proof}
The first inclusion is clear from the definition of $\sA^{\s}$, and
the second holds by Proposition \ref{pro:subcats}.

The claim on the first inclusion:  The first stable t-structure of
Theorem \ref{thm:A_Ktac_SB} shows that $\sA^{\s} = \sA$ is equivalent
to $\sK_{\tac} = 0$.  This happens if and only if each totally acyclic
complex is split exact, that is, if and only if each Gorenstein
projective module is projective.

The claim on the second inclusion:  First, suppose that $\sA^{\s} =
\sK^{(\b)}$.  Let $M$ be an $R$-module with projective resolution $C$;
it follows that $C$ is in $\sA^{\s}$.  Consider the distinguished
triangle $A \rightarrow C \rightarrow U \rightarrow$ with $A$ in $\sA$
and $U$ in $\sK_{\tac}$ which exists by Theorem \ref{thm:A_Ktac_SB}.
Since $U$ is exact, the homology of $A$ is $M$ so the $K$-projective
complex $A$ is a projective resolution of $M$.  This shows that for
each module $M$, the projective resolution is in $\sA$, hence the
Gorenstein projective dimension of $M$ is finite by \cite[thm.\
4.1]{CFH}, and hence $R$ is Gorenstein.

Secondly, suppose that $R$ is Gorenstein and let $C$ be in
$\sK^{(\b)}$.  We will show that $C$ is in $\sA^{\s}$ by showing that
$C$ satisfies the three conditions of Lemma \ref{lem:Kiriko}.

In condition (i), by definition, $C$ has bounded homology.  Since $R$
is Gorenstein, $D$ can be taken to be an injective resolution of $R$.
Hence there is a quasi-isomorphism $R \rightarrow D$ of bounded
complexes, and since $C$ consists of projective modules, it follows
that there is a quasi-isomorphism $R \otimes_R C \rightarrow D
\otimes_R C$.  So $TC = D \otimes_R C$ also has bounded homology.

Conditions (ii) and (iii) hold because the relevant mapping cones are
acyclic, and over a Gorenstein ring this implies that they are
totally acyclic; see \cite[cor.\ (5.5)]{IK}.
\end{proof}

In the following theorem, note that $\sK_{\tac}$ is a triangulated
subcategory of $\sA^{\s}$ which can also be viewed as a triangulated
subcategory of the Verdier quotient $\sA^{\s} / \sK^{\b}$ since there
are only zero morphisms from $\sK^{\b}$ to $\sK_{\tac}$.

\begin{Theorem}
\label{thm:A_Ktac_SB2}
The category $\sA^{\s}/\sK^{\b}$ has a triangle of recollements
\[
  (\sA/\sK^{\b} \, , \, \sK_{\tac} \, , \, S(\sB)/\sK^{\b}).
\]
That is, it has stable t-structures
\[
  (\sA/\sK^{\b} \, , \, \sK_{\tac}) \; , \;
  (\sK_{\tac} \, , \, S(\sB)/\sK^{\b}) \; , \;
  (S(\sB)/\sK^{\b} \, , \, \sA/\sK^{\b}).
\]
\end{Theorem}

\begin{proof}
The first two stable t-structures follow from the stable
t-struc\-tu\-res of Theorem \ref{thm:A_Ktac_SB} by \cite{IKM}. 

Let us show that the third structure exists.  By definition, $\sA^{\s}
= S(\sB) * \sA$, and this implies $\sA^{\s}/\sK^{\b} =
(S(\sB)/\sK^{\b}) * (\sA/\sK^{\b})$.

It is therefore enough to show that each morphism $S(B) \rightarrow A$
in $\sK^{(\b)}/\sK^{\b}$ with $S(B)$ in $S(\sB)/\sK^{\b}$ and $A$ in
$\sA/\sK^{\b}$ must be zero.  Such a morphism is represented by a
diagram $S(B) \rightarrow A^{\prime} \leftarrow A$ in $\sK^{(\b)}$
where the mapping cone of $A \rightarrow A^{\prime}$ is in $\sK^{\b}$.
In particular, the mapping cone is in $\sA$, so $A^{\prime}$ is also
in $\sA$ whence $A^{\prime}$ is isomorphic to a right-bounded complex
of projective modules.  However, $S(B)$ is isomorphic to a
left-bounded complex of projective modules, and it easily follows that
the morphism $S(B) \rightarrow A^{\prime}$ factors through an object
of $\sK^{\b}$.  Hence this morphism becomes zero in
$\sK^{(\b)}/\sK^{\b}$, and so the original morphism $S(B) \rightarrow
A$ in $\sK^{(\b)}/\sK^{\b}$ is zero as desired.
\end{proof}

\section{The category of homomorphisms}
\label{sec:Morph}

This section proves our main result, Theorem A from the Introduction
(Theorem \ref{thm:main}).

\begin{Definition}
We let $\Morph$ denote the category of homomorphisms of $R$-modules.

The objects of $\Morph$ are the homomorphisms of $R$-modules.

The morphisms of $\Morph$ are defined as follows: A morphism $f$ from
$X _\alpha \stackrel{\alpha}{\rightarrow} T_\alpha$ to $X_\beta
\stackrel{\beta}{\rightarrow} T_\beta$ is a pair $(f_X , f_T)$ of
homomorphisms of $R$-modules $X_\alpha \stackrel{f_X}{\rightarrow}
X_\beta$ and $T_\alpha \stackrel{f_T}{\rightarrow} T_\beta$ such that
there is a commutative square
\[
  \xymatrix{
    X_{\alpha} \ar[r]^{f_X} \ar[d]_{\alpha} & X_{\beta} \ar[d]^{\beta} \\
    T_{\alpha} \ar[r]_{f_T} & T_{\beta} \lefteqn{.}
           }
\]
\end{Definition}

\begin{Remark}
\label{rmk:induced_diagram}
Given an object $X _\alpha \stackrel{\alpha}{\rightarrow} T_\alpha$ in
$\Morph$, we will denote the cokernel of $\alpha$ by $N_{\alpha}$.

Observe that a morphism $f$ in $\Morph$ induces a commutative
diagram of $R$-modules with exact rows,
\[
  \xymatrix{
     X_{\alpha} \ar[r]^{\alpha} \ar[d]_{f_X} & T_{\alpha} \ar[r] \ar[d]_{f_T}  & N_{\alpha} \ar[r] \ar[d]^{f_N} & 0 \\
     X_{\beta} \ar[r]_{\beta}  & T_{\beta} \ar[r]   & N_{\beta} \ar[r] & 0\lefteqn{.}
           } 
\] 
\end{Remark}

\begin{Remark}
\label{rmk:pi}
A complex $\pi = \cdots \to \pi ^i \stackrel{d_\pi ^i}{\to} \pi^{i+1}
\cdots$ in $\Morph$ implies a chain map $\pi$ between complexes of
$R$-modules, 
 \[
  \xymatrix{
    \cdots \ar[r]  & X_{\pi ^i} \ar[r] \ar[d]_{\pi ^i} & X_{\pi ^{i+1}} \ar[r] \ar[d]_{\pi ^{i+1}}  & \cdots \\
    \cdots \ar[r]  & T_{\pi ^i} \ar[r] & T_{\pi ^{i+1}} \ar[r]   & \cdots \lefteqn{.}   }
\]
It is not hard to check that the projective objects of $\Morph$ are
precisely the split injections between projective $R$-modules.  Hence,
if $\pi$ is a complex of projective objects in $\Morph$, then there
is an exact sequence
\begin{equation}
\label{equ:XTN}
  0 \to X_\pi \stackrel{\pi}{\to} T_\pi \to N_\pi \to 0
\end{equation}
of complexes of projective $R$-modules.
\end{Remark}

The proof of the following lemma is straightforward.

\begin{Lemma}
\label{dual}
Let $\alpha \stackrel{f}{\rightarrow} \beta$ be a morphism in the
category $\Morph$.  Let $M$ be an $R$-module and consider the zero
homomorphism $0 \stackrel{0^M}{\rightarrow} M$ and the identity $M
\stackrel{1_M}{\rightarrow} M$ as objects of $\Morph$.  Then we have the
following. 

\begin{enumerate}

  \item  There are vertical isomorphisms giving a commutative square 
\[
  \xymatrix {  
    \Hom_R(N_\beta , M) \ar[rrr]^{\Hom_R(f_N , M)} \ar[d]_{\cong} &&&
      \Hom_R(N_\alpha , M) \ar[d]^{\cong} \\
    \Hom_{\Morph}(\beta , 0^M) \ar[rrr]_{\Hom_{\Morph}(f , 0^M)} &&&
      \Hom_{\Morph}(\alpha , 0^M) \lefteqn{.} \\
            }
\]

\smallskip

  \item  There are vertical isomorphisms giving a commutative square 
\[
  \xymatrix {  
    \Hom_R(T_\beta , M) \ar[rrr]^{\Hom_R(f_T , M)} \ar[d]_{\cong} &&&
      \Hom_R(T_\alpha , M) \ar[d]^{\cong} \\
    \Hom_{\Morph}(\beta , 1_M) \ar[rrr]_{\Hom_{\Morph}(f , 1_M)} &&&
      \Hom_{\Morph}(\alpha , 1_M) \lefteqn{.} \\
            }
\]

\end{enumerate}
\end{Lemma}

\begin{Lemma}
\label{tac of Morph}
A complex $\pi$ of projective objects in $\Morph$ is totally acyclic
if and only if each of the complexes
\begin{align*}
  X_\pi & = \cdots 
            \longrightarrow X_{\pi ^i}
            \longrightarrow X_{\pi ^{i+1}}
            \longrightarrow \cdots, \\
  T_\pi & = \cdots
            \longrightarrow T_{\pi ^i}
            \longrightarrow T_{\pi ^{i+1}}
            \longrightarrow \cdots
\end{align*}
belongs to $\sK_{\tac}$.
\end{Lemma}

\begin{proof}
Let $\varphi$ be a projective object of $\Morph$.  Remark \ref{rmk:pi}
says that $\varphi$ is a split injection of projective $R$-modules, so
there are projective $R$-modules $P$ and $P^{\prime}$ such that
$\varphi = 0^P \oplus 1_{P^{\prime}}$.  The complex $\Hom_{\Morph}(\pi
, \varphi)$ is acyclic if and only if both $\Hom_{\Morph}(\pi , 0^P)$
and $\Hom_{\Morph}(\pi , 1_{P^{\prime}})$ are acyclic.  By Lemma
\ref{dual}, this is equivalent to having both complexes $\Hom_{R}(T_\pi , P)$
and $\Hom_{R}(N_\pi , P^{\prime})$ acyclic.

Therefore $\pi$ is totally acyclic if and only if $T_\pi$ and $N_\pi$
are both totally acyclic, which by the sequence \eqref{equ:XTN} is
equivalent to both of $T_\pi$ and $X_\pi$ being totally acyclic.
\end{proof}

\begin{Corollary}
The Gorenstein projective objects of $\Morph$ are the injective
homomorphisms between Gorenstein projective $R$-modules which have
Gorenstein projective cokernels.
\end{Corollary}

\begin{proof}  
A Gorenstein projective object in $\Morph$ is a cycle of a totally
acyclic complex of projective objects of $\Morph$.  It follows
easily from Lemma \ref{tac of Morph} that it is an injective
homomorphism between Gorenstein projective $R$-modules, and that the
cokernel is Gorenstein projective.

Conversely, let $X_\alpha$ and $T_\alpha$ be Gorenstein projective
$R$-modules and suppose that $X_\alpha \stackrel{\alpha}{\to}
T_\alpha$ is an injective homomorphism with Gorenstein projective
cokernel.  Using the Horseshoe Lemma, the short exact sequence $0
\rightarrow X_{\alpha} \stackrel{\alpha}{\rightarrow} T_{\alpha}
\rightarrow N_{\alpha} \rightarrow 0$ gives a short exact sequence of
complete projective resolutions
\[
  0 \longrightarrow P_{X_\alpha}
    \stackrel{\pi_\alpha}{\longrightarrow} P_{T_\alpha}
    \longrightarrow P_{N_{\alpha}}
    \longrightarrow 0.
\]
Lemma \ref{tac of Morph} says that $P_{X_\alpha}
\stackrel{\pi_\alpha}{\longrightarrow} P_{T_\alpha}$ can be viewed as a
totally acyclic complex of projective objects of $\Morph$, and it
is clear that it is a complete projective resolution of $X_\alpha
\stackrel{\alpha}{\to} T_\alpha$ which is hence a Gorenstein
projective object of $\Morph$.
\end{proof} 

\begin{Definition}
We denote the full subcategory of Gorenstein projective objects in
$\Morph$ by $\GMor$.  Inside $\GMor$, we consider the
following full subcategories $\GMor^p$, $\GMor^0$, and
$\GMor^1$. 
\begin{enumerate}

  \item $\GMor^p$ consists of injective homomorphisms $X
  \stackrel{\iota_X}{\rightarrow} P$ where $X$ is Gorenstein
  projective and $P$ is projective.
  
\smallskip

  \item $\GMor^0$ consists of zero homomorphisms $0
  \stackrel{0^T}{\rightarrow} T$ where $T$ is Gorenstein projective.
  
\smallskip

  \item $\GMor^1$ consists of identity homomorphisms $X
  \stackrel{1_M}{\rightarrow} X$ where $X$ is Gorenstein projective.

\end{enumerate}
There are corresponding stable categories which are defined by
di\-vi\-ding out the morphisms which factor through a projective
object.  The stable categories are denoted by underlining.  The
category $\underline{\GMor}$ is triangulated, and
$\underline{\GMor}^p$, $\underline{\GMor}^0$, and
$\underline{\GMor}^1$ are triangulated subcategories.
\end{Definition}

\begin{Theorem} 
The category $\underline{\GMor}$ has a triangle of recollements
\[
  (\underline{\GMor}^p
   \, , \, \underline{\GMor}^1
   \, , \, \underline{\GMor}^0).
\]
That is, it has stable t-structures 
\[
  (\underline{\GMor}^p \, , \, \underline{\GMor}^1) \; , \;
  (\underline{\GMor}^1 \, , \, \underline{\GMor}^0) \; , \;
  (\underline{\GMor}^0 \, , \, \underline{\GMor}^p).
\]
\end{Theorem}

\begin{proof} 
It is enough to show that each of the following categories:  (i)
$\underline{\GMor}^p * \underline{\GMor}^1$, (ii) $\underline{\GMor}^1
* \underline{\GMor}^0$, and (iii) $\underline{\GMor}^0 *
\underline{\GMor}^p$ is equal to $\underline{\GMor}$.

Let $X_{\alpha} \stackrel{\alpha}{\to} T_{\alpha}$ be an object of
$\GMor$ and consider the exact sequence $0 \to X_{\alpha}
\stackrel{\alpha}{\to} T_{\alpha} \stackrel{\beta}{\to} N_\alpha \to
0$ of G-projective $R$-modules.  There exist injective homomorphisms
$\iota _{T_\alpha} : T_\alpha \to P$ and $\iota _{N_\alpha} :
N_\alpha \to P'$ with projective $R$-modules $P$ and $P'$.

(i) The commutative diagram with exact rows 
\[
  \xymatrix{
     0 \ar[rr] && X_{\alpha} \ar[rr]^{\alpha} \ar[d]_{\alpha} && T_{\alpha} \ar[rr]^{\beta} \ar[d]_{1 \choose 0} && N_\alpha \ar[rr] \ar[d]^{\iota_{N_\alpha}} && 0\\
     0 \ar[rr] && T_{\alpha} \ar[rr]_-{1 \choose -\iota_{N_{\alpha}}\beta} && T_{\alpha} \oplus P' \ar[rr]_-{(\iota_{N_{\alpha}}\beta\, , \,1)} && P' \ar[rr] && 0 \\
           } 
    \] 
induces a distinguished triangle in $\underline{\GMor}$  
\[
  \Sigma^{-1}\underline{\iota _{N_\alpha}}
  \to \underline{\alpha}  
  \to \underline{1_{T_\alpha}}
  \to
\]
with $\Sigma^{-1}\underline{\iota _{N_\alpha}}$ in
$\underline{\GMor}^p$ and $\underline{1_{T_\alpha}}$ in
$\underline{\GMor}^1$. 

(ii)  The commutative diagram with exact rows  
\[
  \xymatrix{
     0 \ar[r] & X_{\alpha} \ar@{=}[r] \ar@{=}[d] & X_{\alpha} \ar[r] \ar[d]_{\alpha}  & 0 \ar[r] \ar[d]^{0^{N_\alpha}} & 0\\
     0 \ar[r] & X_{\alpha} \ar[r]_{\alpha}  & T_\alpha \ar[r]_{\beta}   & N_\alpha \ar[r]  & 0\\
        } 
    \] 
induces a distinguished triangle in $\underline{\GMor}$  
\[
  \underline{1_{X_\alpha}}
  \to \underline{\alpha}
  \to \underline{0^{N_\alpha}}
  \to
\]
with $\underline{1_{X_\alpha}}$ in $\underline{\GMor}^1$ and
$\underline{0^{N_\alpha}}$ in $\underline{\GMor}^0$.

(iii) The commutative diagram with exact rows  
\[
  \xymatrix{
     0 \ar[r] & X_{\alpha} \ar@{=}[r] \ar[d]_{\alpha} & X_{\alpha} \ar[r] \ar[d]_{\iota _{T_\alpha} \alpha}  & 0 \ar[r] \ar[d]^{0^{\Sigma T_\alpha}} & 0\\
     0 \ar[r] & T_{\alpha} \ar[r]_{\iota _{T_{\alpha}}}  & P \ar[r]   & \Sigma T_\alpha \ar[r]  & 0\\
        } 
    \] 
induces a distinguished triangle in $\underline{\GMor}$  
\[
  \underline{0^{T_\alpha}}
  \to \underline{\alpha}
  \to \underline{\iota _{T_\alpha} \alpha}
  \to
\]
with $\underline{0^{T_\alpha}}$ in $\underline{\GMor}^0$ and
$\underline{\iota _{T_\alpha} \alpha}$ in $\underline{\GMor}^p$.
\end{proof} 

Let $X_{\alpha} \stackrel{\alpha}{\to} T_{\alpha}$ be an object of
$\GMor$ and consider complete projective resolutions $P$ of
$X_{\alpha}$ and $\widetilde{P}$ of $T_{\alpha}$.  In particular,
there is a surjection $P^0 \stackrel{\rho}{\rightarrow} X_{\alpha}$
and an injection $T_{\alpha} \stackrel{\iota}{\rightarrow}
\widetilde{P}^1$.  Let $P_{\alpha}$ denote the complex
\[
  \cdots
  \longrightarrow P^{-1}
  \longrightarrow P^0
  \stackrel{\iota\alpha\rho}{\longrightarrow} \widetilde{P}^1
  \longrightarrow \widetilde{P}^2
  \longrightarrow \cdots.
\]

\begin{Proposition}
[{\cite[lemmas 4.2 and 4.3 and prop.\ 4.4]{IKM}}]
The operation $\alpha \mapsto P_{\alpha}$ gives a functor $\GMor \to
\sA^s$ which induces a triangulated functor
\[
  \underline{P} : \underline{\GMor} \to \sA^s / \sK^b.
\]
\end{Proposition}

\begin{Lemma}
[{\cite[lemmas 4.6 and 4.7]{IKM}}]
\label{lem:IKM_4.5_and_4.6}
\begin{enumerate}

  \item $\underline{P}$ sends the triangle of recollements
\[
  (\underline{\GMor}^p \, , \, \underline{\GMor}^1 \, , \,
  \underline{\GMor}^0)
\]
to the triangle of re\-col\-le\-ments
\[
  ( \sA / \sK^b \, , \, \sK_{\tac} \, , \, S(\sB) / \sK^b ).
\]
      
\smallskip

  \item The restriction of $\underline{P}$ to $\underline{\GMor}^1$
  is an equivalence of triangulated categories $\underline{\GMor}^1
  \to \sK_{\tac}$.
   
\end{enumerate}
\end{Lemma} 

\begin{Proposition}
[{\cite[Prop.\ 1.18]{IKM}}]
\label{pro:IKM_1.18}
Let $(\sU, \sV, \sW)$ and $(\sU', \sV', \sW')$ be triangles of
recollements in $\sT$ and $\sT'$ respectively.  Suppose the
triangulated functor $F: \sT \to \sT'$ sends $(\sU, \sV, \sW)$ to
$(\sU', \sV', \sW')$.  If the restriction $F \mid \sU$ is an
equivalence of triangulated categories, then so is $F$. 
\end{Proposition}

The following main theorem follows immediately by combining Lemma
\ref{lem:IKM_4.5_and_4.6} and Proposition \ref{pro:IKM_1.18}; compare
with \cite[lem.\ 4.7 and thm.\ 4.8]{IKM}.

\begin{Theorem} 
\label{thm:main}
The functor $\underline{P}$ is an equivalence of triangulated
categories
\[
  \underline{\GMor} \to \sA^s / \sK^b.
\]
\end{Theorem}

\end{document}